\documentclass{daj}

\dajAUTHORdetails{%
  title = {The Sharp Square Function Estimate with Matrix Weight}, 
  author = {Tuomas Hyt\"onen, Stefanie Petermichl, and Alexander Volberg},
  plaintextauthor = {Tuomas Hytonen, Stefanie Petermichl, and Alexander Volberg},
  plaintexttitle = {The Sharp Square Function Estimate with Matrix Weight}, 
}   

\dajEDITORdetails{%
   year={2019},
   number={2},
   received={15 May 2018},   
   published={28 March 2019},  
   doi={10.19086/da.7597},       
}   

\usepackage[english]{babel}
\usepackage{amsmath,amssymb,bbm,latexsym}
\newcommand{\mathd}{\mathrm{d}}
\newcommand{\tmop}[1]{\ensuremath{\operatorname{#1}}}
\newcommand{\tmtextit}[1]{{\itshape{#1}}}
\newenvironment{proof}{\noindent\textbf{Proof\ }}{\hspace*{\fill}$\Box$\medskip}
\newtheorem{theorem}{Theorem}
\newtheorem{proposition}{Proposition}

\begin{document}

\begin{frontmatter}[classification=text]

\title{The Sharp Square Function Estimate with Matrix Weight} 

\author[T.H.]{Tuomas Hyt\"onen\thanks{Supported by the Finnish Centre of Excellence in Analysis and Dynamics Research}}
\author[S.P.]{Stefanie Petermichl\thanks{Supported by the ERC project CHRiSHarMa DLV-862402}}
\author[A.V.]{Alexander Volberg\thanks{Supported by the NSF grant DMS-160065}}



%
%
%


\begin{abstract}
  We prove the matrix $A_2$ conjecture for the dyadic square function, that
  is, an estimate of the form
  \[ \| S_W \|_{L_{\mathbb{C}^d}^2(W) \rightarrow L_{\mathbb{R}}^2} \lesssim
     [W]_{A_2}, \]
  where the focus is on the sharp linear dependence on the matrix $A_2$ constant.
  Moreover, we give a mixed estimate in terms of $A_2$ and $A_{\infty}$ constants. The key
  to the proof is a sparse domination of a process inspired by the integrated form of the
  matrix--weighted square function.
\end{abstract}
\end{frontmatter}

\section{Introduction}

The theory of weights has drawn much attention in recent years. In the
scalar--valued setting, we say that a non--negative locally integrable
function is a dyadic $A_2$ weight iff
\[ [w]_{A_2} = \sup_I \langle w \rangle_I \langle w^{- 1} \rangle_I < \infty ,
\]
where the supremum runs over dyadic intervals and $\langle \cdot \rangle_I$
returns the average of a function over the interval $I$. It is classical that
this is the necessary and sufficient condition for the dyadic square function,
maximal function, Hilbert transform and Calder\'on--Zygmund operators to be
bounded on
\[ L_{\mathbb{R}}^2 (w) = \left\{ f : \| f \|^2_{L_{\mathbb{R}}^2 (w)} =
   \int | f |^2 w < \infty \right\} . \]

Partly inspired by applications to PDE, but also interesting in their own
right, the precise dependence on the $A_2$ characteristic for the different
operators has been under intensive investigation -- these questions became
known as $A_2$ conjectures. Most such improved estimates came at the cost and
benefit of an array of fantastic new ideas and techniques in harmonic
analysis. The first such example was the sharp weighted estimate of the dyadic
square function by Hukovic--Treil--Volberg {\cite{HTV}} followed by the
martingale multiplier by Wittwer {\cite{W}}, the Beurling operator by
Petermichl--Volberg {\cite{PV}}, the Hilbert and Riesz transforms by
Petermichl {\cite{Philbert, Priesz}} and all Calder\'on--Zygmund
operators by Hyt\"onen {\cite{HallCZO}}. Somewhat later it has been discovered
that many of these estimates can be slightly improved by replacing a half
power of these norm estimates by a half power of the smaller $A_{\infty}$
norm, the best estimate $C$ in the inequality
\[ \langle M_I w \rangle_I \leqslant C \langle w \rangle_I ,\]
where $M_I$ is the dyadic maximal function localized to $I$, see
{\cite{HP, HL, LL}}. The field has since seen beautiful new proofs of these
optimal results {\cite{Le1, La, HPTV, LPR, T}} and
many extensions far beyond Calder\'on--Zygmund theory
{\cite{HRT, BFP, CDO, Le, BBL}}.

\

Inspired by applications to multivariate stationary stochastic processes, a
theory of matrix weights was developped by Treil and Volberg
{\cite{TVangle}}, where the necessary and
sufficient condition for boundedness of the Hilbert transform was found, the
matrix $A_2$ characteristic. Aside from an early, excellent estimate for a
natural maximal function in this setting by Christ--Goldberg
{\cite{CG, IKP}}, the optimal norm estimates for singular operators
seemed out of reach. Some improvement was achieved in Bickel--Petermichl--Wick
{\cite{BPW}} with a new estimate for Hilbert and martingale transforms of
$[W]^{3 / 2}_{A_2} \log (1 + [ W ]_{A_2})$. Recently, by
Nazarov--Petermichl--Treil--Volberg {\cite{NPTV}} the logarithmic term was
dropped and it was shown that for all Calder\'on--Zygmund operators there holds
an estimate of the order $[W]^{3 / 2}_{A_2}$.

In {\cite{BPW}} a notable improvement to the dyadic square function estimates of {\cite{TVangle, Vmatrix}} was given, namely
\[ \| S_W \|_{L_{\mathbb{C}^d}^2(W) \rightarrow L_{\mathbb{R}}^2} \lesssim [W]_{A_2} \log (1 + [ W ]_{A_2}). \]
The above estimate only features an extra logarithmic term as compared to the
sharp, linear estimate in the scalar case {\cite{HTV}}. Despite the advance on
the matrix weighted Carleson lemma in {\cite{CNT}}, a
natural tool for square function estimates in the scalar setting, the
logarithmic term could not be removed. One of the many difficulties arising,
stem from the non--commutativity and it seems that most convex functions are
not matrix convex. In this paper, we remove the logarithmic term by other
means and give the first sharp estimate of a singular operator in the matrix
weighted setting:
\[ \| S_W \|_{L_{\mathbb{C}^d}^2(W) \rightarrow L_{\mathbb{R}}^2} \lesssim [W]_{A_2} . \]
This estimate is known to be optimal among all upper bounds of the form $\phi([W]_{A_2})$ even in the scalar setting.
The scalar example for sharpness directly implies matrix examples of all dimensions by considering weights of the form $w\otimes I_d$, where $I_d$ is the identity matrix in dimension $d$. Allowing for a more general dependence on the weight $W$, we even show that
\[ \| S_W \|_{L_{\mathbb{C}^d}^2(W) \rightarrow L_{\mathbb{R}}^2} \lesssim [W]^{1 / 2}_{A_2} [W^{-1}]^{1 / 2}_{A_{\infty}} \]
using the matrix $A_{\infty}$ characteristic. 
This coincides with the mixed $A_2$--$A_\infty$ bound in the scalar case obtained in \cite{LL}.

\section{Notation}

Let $\mathcal{D}$ be the collection of dyadic subintervals of $J = [0, 1]$.
We call a $d \times d$ matrix--valued function
$W$ a weight, if $W$ is entry--wise locally integrable and if $W (x)$ is
positive semidefinite almost everywhere. One defines $L_{\mathbb{C}^d}^2 (W)$
to be the space of vector functions with
\[ \| f \|^2_{L_{\mathbb{C}^d}^2 (W)} = \int_J \| W^{1 / 2} (x) f (x)
   \|_{\mathbb{C}^d}^2 \mathd x = \int_J \langle W (x) f (x), f (x)
   \rangle_{\mathbb{C}^d} \mathd x < \infty . \]
The dyadic matrix Muckenhoupt $A_2$ condition is
\[ [W]_{A_2} = \sup_{I \in \mathcal{D}} \| \langle W \rangle^{1 / 2}_I \langle
   W^{- 1} \rangle_I^{1 / 2} \|^2 < \infty , \] where we mean the operator norm.  
The dyadic matrix $A_{\infty}$ condition is
\[ [W]_{A_{\infty}} = \sup_{e \in \mathbb{C}^d} [W_e]_{A_{\infty}} \]
where $W_e (x) = \langle W (x) e, e \rangle_{\mathbb{C}^d}.$ Note that the norm of $e$ is irrelevant in the definition, since a constant multiple of a weight has the same $A_\infty$ norm. After introducing matrix $A_\infty$ and beforeLet us recall
the fact that $[W]_{A_{\infty}} \lesssim [W]_{A_2}$; see \cite{NPTV}, Section 4.


\

Let $h_I = | I |^{-1 / 2} (\chi_{I_+} - \chi_{I_-})$ be the
$L_{\mathbb{R}}^2$ normlized Haar function and let with $\sigma_I = \pm 1$
\[ T_{\sigma} f (x) = \sum_{I \in \mathcal{D}} \sigma_I (f, h_I) h_I (x), \]
defined both on scalar--valued as well as vector--valued functions $f$. Recall
that the dyadic square function for real--valued, mean zero functions $f$ is
defined as
\[ S f (x)^2 = \sum_I | (f, h_I) |^2 \frac{\chi_I (x)}{| I |} . \]
Its classical vector analog becomes the scalar--valued function
\[ S f (x)^2 = \sum_I \| (f, h_I) \|_{\mathbb{C}^d}^2 \frac{\chi_I (x)}{| I
   |} . \]
When working with matrix weights, it is customary to include the weight in the
definition of the (scalar--valued) operator, such as done for example by
Christ--Goldberg {\cite{CG}} for the maximal function
\[ M_W f (x) = \sup_{I : x \in I} \frac{1}{| I |} \int_I \| W^{1 / 2} (x) W^{-
   1 / 2} (y) f (y) \|_{\mathbb{C}^d} \mathd y. \]
Recall that $S_W : L^2_{\mathbb{C}^d}(W) \rightarrow L^2_{\mathbb{R}}$ is
defined by (see {\cite{PPburkholder}})
\[ S_W f (x)^2 =\mathbb{E} (\| W (x)^{1 / 2} T_{\sigma} f (x)
   \|^2_{\mathbb{C}^d}), \]
 where $\mathbb{E}$ is the expectation over independent uniformly distributed random signs $\sigma_I=\pm 1$.
One calculates
\begin{eqnarray*}
  \| S_W f \|_{L_{\mathbb{R}}^2}^2 
  & = & 
  \int_J \mathbb{E} \left( \sum_{I,  I'} \sigma_I \sigma_{I'} h_I (x) h_{I'} (x) 
  \left\langle W (x) (f, h_I), (f, h_{I'}) \right\rangle_{\mathbb{C}^d} \right) \mathd x\\
  & = & \int_J \sum_{I, I'} \mathbb{E} (\sigma_I \sigma_{I'}) h_I (x) h_{I'}
  (x) \left\langle W (x) (f, h_I), (f, h_{I'}) \right\rangle_{\mathbb{C}^d} \mathd x\\
  & = & \sum_I \left\langle \langle W \rangle_I (f, h_I), (f, h_I)
  \right\rangle_{\mathbb{C}^d} .
\end{eqnarray*}
The study of these sums was introduced by Volberg in {\cite{Vmatrix}}. Indeed,
in the scalar setting, this square function $S_w$ is bounded into the unweighted
$L^2$ if and only if the classical dyadic square function is bounded into the
weighted $L^2 (w)$:
\[ \| S_w f \|_{L_{\mathbb{R}}^2}^2 = \sum_I \langle w \rangle_I | (f, h_I)
   |^2 = \| S f \|_{L_{\mathbb{R}}^2 (w)}^2 . \]

\section{Results}

Here is our main theorem.

\begin{theorem}\label{thm:A2}
  \[ S_W : L^2_{\mathbb{C}^d}(W) \rightarrow L_{\mathbb{R}}^2 \]
  has operator norm bounded by a constant multiple of
  \[ [W]^{1 / 2}_{A_2} [W^{-1}]^{1 / 2}_{A_{\infty}} \lesssim [W]^1_{A_2} . \]
  This estimate is sharp among all upper bounds of the form $\phi([W]_{A_2})$.
\end{theorem}

The previously known best estimate {\cite{BPW}} was $\| S_W \|_{L^2_{\mathbb{C}^d}(W) \rightarrow L_{\mathbb{R}}^2} \lesssim
[W]_{A_2} \log (1 + [W]_{A_2})$. With a different method we drop the
logarithmic term and improve the single power $A_2$ constant to split into a
$A_2 - A_{\infty}$ estimate.

A key to the proof is the following sparse domination result of independent interest. Recall that a collection of intervals $\mathcal S\subset\mathcal D$ is called sparse, if there are disjoint subsets $E(I)\subset I$ for every $I\in\mathcal S$ such that $|E(I)|\geq\frac12|I|$.

\begin{proposition}\label{prop:sparse}
Given $f\in L^2_{\mathbb{C}^d}(W)$, there exists a sparse collection $\mathcal S\subset\mathcal D$ such that
\[
  \|S_W f\|_{L^2_{\mathbb{R}}}^2
  \lesssim \sum_{I\in\mathcal S} \langle \| \langle W
     \rangle^{1 / 2}_I f \|_{\mathbb{C}^d} \rangle_I^2 |I|
    =\Big\|\Big( \sum_{I\in\mathcal S} \langle \| \langle W
     \rangle^{1 / 2}_I f \|_{\mathbb{C}^d} \rangle_I^2 \chi_I\Big)^{1/2}\Big\|_{L^2_{\mathbb{R}}}^2.
\]
\end{proposition}

\begin{proof}[Proof of the Theorem assuming the Proposition]
  Sharpness follows from the scalar case, as explained in the introduction. Let us attend to the upper estimate
  $\|S_W f\|_{L_{\mathbb{R}}^2 }\lesssim [W]^{1 / 2}_{A_2} [W^{-1}]^{1 / 2}_{A_{\infty}}\|f\|_{L^2_{\mathbb{C}^d}(W)}$.
  Substituting $W^{-1/2}f$ in place of $f$ and using the bound from the Proposition, we should show that
  \[
     \Big\|\Big( \sum_{I\in\mathcal S} \langle \| \langle W
     \rangle^{1 / 2}_I W^{-1/2} f \|_{\mathbb{C}^d} \rangle_I^2 \chi_I\Big)^{1/2}\Big\|_{L^2_{\mathbb{R}}}
     \lesssim [W]_{A_2}^{1/2} [W^{-1}]_{A_{\infty}}^{1/2}\|f\|_{L^2_{\mathbb{C}^d}}.
  \]
  But the left hand side is the $L^2_{\mathbb{R}}$ norm of 
  the sparse square function $S_{3,W}f$ defined in {\cite{NPTV}},
  for which it was proved {\cite{NPTV}} that
  \[ \| S_{3, W} \|_{L_{\mathbb{C}^d}^2 \rightarrow L_{\mathbb{R}}^2}
   \lesssim [W]^{1 / 2}_{A_2} [W^{-1}]^{1 / 2}_{A_{\infty}}.
  \]
  And this is exactly the estimate we needed.
\end{proof}
  
\begin{proof}[Proof of the Proposition]
%
  Consider the collection of first stopping intervals $L\subset J$ determined by
  \[ \| \langle W \rangle^{1 / 2}_L \langle W \rangle^{- 1 / 2}_J \| > C_1 \]
  or
  \[ \sum_{I : I \supseteqq L} \| \langle W \rangle^{1 / 2}_J (f, h_I)
     \|_{\mathbb{C}^d}^2 \frac{1}{| I |} > C_2 \langle \| \langle W
     \rangle^{1 / 2}_J f \|_{\mathbb{C}^d} \rangle^2_J \]
  but for all $L' \supsetneqq L$
  \[ \| \langle W \rangle^{1 / 2}_{L'} \langle W \rangle^{- 1 / 2}_J \|
     \leqslant C_1 \]
  and
  \[ \sum_{I : I \supseteqq L'} \| \langle W \rangle^{1 / 2}_J (f, h_I)
     \|_{\mathbb{C}^d}^2 \frac{1}{| I |} \leqslant C_2 \langle \| \langle W
     \rangle^{1 / 2}_J f \|_{\mathbb{C}^d} \rangle^2_J . \]
  Consider $\mathcal{S}^1$, the collection of all first stopping intervals.
  Our final sum splits after this step
  \begin{equation*}
\begin{split}
   \sum_I  &\ \| \langle W \rangle^{1 / 2}_I (f, h_I) \|^2_{\mathbb{C}^d} \\
   &=\sum_{I : \forall L \in \mathcal{S}^1 : I \varsubsetneqq L} \| \langle W
     \rangle^{1 / 2}_I (f, h_I) \|^2_{\mathbb{C}^d} + \sum_{I : \exists L \in
     \mathcal{S}^1 : I \subseteqq L} \| \langle W \rangle^{1 / 2}_I (f, h_I)
     \|^2_{\mathbb{C}^d} .
\end{split}
\end{equation*}
  We estimate the first sum.
  \begin{eqnarray*}
    \sum_{I : \forall L \in \mathcal{S}^1 : I \varsubsetneqq L} &  &  \langle
    \langle W \rangle_I (f, h_I), (f, h_I) \rangle_{\mathbb{C}^d} \\
    & \leqslant
    & \sum_{I : \forall L \in \mathcal{S}^1 : I \varsubsetneqq L} \| \langle W
    \rangle^{1 / 2}_I \langle W \rangle_J^{- 1 / 2} \|^2  \langle \langle W
    \rangle_J (f, h_I), (f, h_I) \rangle_{\mathbb{C}^d}\\
    & \leqslant & C^2_1 \sum_{I : \forall L \in \mathcal{S}^1 : I
    \varsubsetneqq L} \langle \langle W \rangle_J (f, h_I), (f, h_I)
    \rangle_{\mathbb{C}^d}\\
    & = & C^2_1 \int_J \sum_{I : \forall L \in \mathcal{S}^1 : I
    \varsubsetneqq L} \| \langle W \rangle^{1 / 2}_J (f, h_I)
    \|_{\mathbb{C}^d}^2 \frac{\chi_I (x)}{| I |} \mathd x\\
    & \leqslant & C^2_1 C_2 \langle \| \langle W \rangle^{1 / 2}_J f
    \|_{\mathbb{C}^d} \rangle^2_J | J | .
  \end{eqnarray*}
  The first step is a triangle inequality for norms, the second step uses the
  first stopping condition and the last step uses the second stopping
  condition and a resulting pointwise estimate.
    By iteration we have the domination
  \[ \sum_I \langle \langle W \rangle_I (f, h_I), (f, h_I)
     \rangle_{\mathbb{C}^d} \leqslant C^2_1 C_2 \sum_{L \in \mathcal{S}}
     \langle \| \langle W \rangle^{1 / 2}_L f \|_{\mathbb{C}^d} \rangle^2_L |
     L |, \]
  precisely the integrated form of $S_{3, W}$, provided $\mathcal{S}$ is
  sparse.
  
  It remains to show the collection $\mathcal{S}$ is sparse for large enough
  $C_1$ and $C_2$. The collection stemming from
  \[ \sum_{I : L \subseteqq I \subseteqq J} \| \langle W \rangle^{1 / 2}_J (f,
     h_I) \|_{\mathbb{C}^d}^2 \frac{1}{| I |} > C_2 \langle \| \langle W
     \rangle^{1 / 2}_J f \|_{\mathbb{C}^d} \rangle^2_J \]
  is sparse for large enough $C_2$ because of the (unweighted) weak type boundedness of $S:L^1_{\mathbb{C}^d}\to L^1_{\mathbb{R}}$ (with a universal constant), where $S$ is the
  standard square function
  \[
     Sg(x):=\sum_{I\in\mathcal D} \| (g,
     h_I) \|_{\mathbb{C}^d}^2 \frac{\chi_I(x)}{| I |},
  \]
  now applied to $g=\langle W \rangle^{1 / 2}_J f$; indeed, the above stopping condition means that $Sg(x)>\sqrt{C_2} \langle \|g \|_{\mathbb{C}^d} \rangle_J$ for all $x\in L$ and thus the union of these intervals $L$ is contained in $\{Sg>\sqrt{C_2} \langle \| g \|_{\mathbb{C}^d} \rangle_J\}$, which has measure at most $C_2^{-1/2}\|S\|_{L^1_{\mathbb{C}^d}\to L^1_{\mathbb{R}}}$.
  
   The collection stemming from $\| \langle W \rangle^{1 /
  2}_L \langle W \rangle^{- 1 / 2}_J \| > C_1$ is sparse because
  \begin{eqnarray*}
    C_1^2 \sum_{L } | L |
    & < & \sum_{L } | L | \| \langle W \rangle^{1 / 2}_L \langle W \rangle^{- 1 / 2}_J \|^2  \\
    & \leq & \sum_{L } | L | \| \langle W \rangle^{1 / 2}_L \langle W \rangle^{- 1 / 2}_J \|_{S_2}^2  \\
    & = & \sum_L | L | \tmop{tr} (\langle W \rangle^{- 1 / 2}_J \langle W
    \rangle_L \langle W \rangle^{- 1 / 2}_J)\\
    & = & \sum_L \int_L \tmop{tr} (\langle W \rangle^{- 1 / 2}_J W (x)
    \langle W \rangle^{- 1 / 2}_J) \mathd x\\
    & \leqslant & \int_J \tmop{tr} (\langle W \rangle^{- 1 / 2}_J W (x) \langle W \rangle^{- 1 / 2}_J) \mathd x\\
    & = & \tmop{tr} (I_d) |J|  = d | J |.
  \end{eqnarray*}
  The first inequality uses the first stopping condition, then we dominate the operator norm by the Hilbert Schmidt norm (denoted by $\| \cdot \|_{S_2}$). In the sequel we use the linearity of trace and disjointness of stopping intervals.
\end{proof}

\section*{Acknowledgments}
\thanks{
This research was conducted during the authors' NSF-supported participation in the Spring 2017 Harmonic Analysis Program at the Mathematical Sciences Research Institute in Berkeley, California. 

\bibliographystyle{amsplain}

\

%
%
%

\begin{dajauthors}
\begin{authorinfo}[T.H.]
  Tuomas Hyt\"onen\\
  University of Helsinki\\
  Helsinki, Finland\\
  tuomas.hytonen\imageat{}helsinki\imagedot{}fi \\
 \end{authorinfo}
\begin{authorinfo}[S.P.]
  Stefanie Petermichl\\
  Universit\'e Paul Sabatier\\
  Toulouse, France\\
  stefanie.petermichl\imageat{}math.univ-toulouse\imagedot{}fr \\
\end{authorinfo}
\begin{authorinfo}[A.V.]
  Alexander Volberg\\
  Michigan Sate University\\
  East Lansing, MI, U.S.A\\
  volberg\imageat{}math\imagedot{}msu\imagedot{}edu\\
 \end{authorinfo}
\end{dajauthors}

\end{document}